\documentclass[11pt,a4paper]{article}
\usepackage[T1]{fontenc}
\usepackage{amsmath,amssymb,amsthm,amsfonts}
\usepackage{xcolor}
\usepackage{booktabs}
\usepackage{float}
\usepackage[colorlinks=true,linkcolor=black,citecolor=black,urlcolor=blue]{hyperref}
\usepackage{geometry}
\newcommand{\new}[1]{#1}
\newcommand{\B}[1]{#1}
\newcommand{\C}[1]{#1}
\newcommand{\K}[1]{#1}
\newcommand{\V}[1]{#1}

\newtheorem{theorem}{Theorem}[section]
\newtheorem{proposition}[theorem]{Proposition}
\newtheorem{lemma}[theorem]{Lemma}
\newtheorem{corollary}[theorem]{Corollary}

\theoremstyle{definition}
\newtheorem{remark}[theorem]{Remark}
\newtheorem{question}[theorem]{Question}

\title{\textbf{Infinite families and sporadic solutions\\ of the equation $\varphi(n)=c+\alpha\sum_{p\mid n}\varphi(p)$}}

\author{Hac\`ene Belbachir\thanks{Corresponding author. ORCID: 0000-0001-8540-3033.}\\[2pt]
\small USTHB, Faculty of Mathematics, RECITS Laboratory,\\
\small Po. Box 32 El Alia, 16111 Bab Ezzouar, Algiers, Algeria\\
\small \texttt{hacenebelbachir@gmail.com}, \texttt{hbelbachir@usthb.dz}
\and
\B{Ravi Ranjan}\\[2pt]
\small \B{Department of Mathematics, Central University of South Bihar,}\\
\small \B{Gaya, Bihar 824236, India}\\
\small \K{\texttt{raviranjan23@cusb.ac.in}}
\and
Shubh N. Singh\\[2pt]
\small Department of Mathematics, Central University of South Bihar,\\
\small Gaya, Bihar 824236, India\\
\small \texttt{shubh@cub.ac.in}}

\date{October 23, 2026}

\begin{document}
\maketitle

\begin{abstract}
\noindent
Ranjan and Singh proved that for every composite $n$ which is not twice an odd prime, $\varphi(n)$ is at least one plus the sum of $\varphi$ over the prime divisors of $n$, with equality exactly for $n=4,12,30$. \new{We study the weighted equation $\varphi(n)=c+\alpha\sum_{p\mid n}\varphi(p)$. We determine the parameters $(c,\alpha)$ admitting infinite families of solutions, prove that all remaining solutions are finite in number and arise from an explicit linear mechanism, and compute them exactly in a range of parameters; the finiteness of $\{4,12,30\}$ becomes a special case. We also settle the analogous equations for the divisor functions $\tau$ and $\sigma$ and for the Jordan totients, and, from known results on the sum of the prime divisors, derive the full asymptotic expansion of the average deficiency.} \K{In the range $|c|\le10$, $1\le\alpha\le6$ the composite solution set is determined completely\V{, and the number and the size of the sporadic solutions are shown to be bounded by polynomials in $|c|+\alpha$}.} \B{Finally, replacing the prime values $\varphi(p)$ by their iterates $\varphi^{(k)}(p)$, we determine all solutions of $\varphi(n)=1+\sum_{p\mid n}\varphi^{(k)}(p)$: the solution sets are $\{1,4,12,30\}$, $\{1,3,4,10\}$ and, for every $k\ge3$, $\{1,3,4\}$; for the weighted iterated equation we exhibit an infinite family governed by the safe primes.}
\medskip

\noindent\textbf{2020 MSC:} 11A25 (primary); 11N37 (secondary).\\
\textbf{Keywords:} Euler's totient function, divisor functions, Jordan totient, prime--composite characterizations, average orders\B{, iterated totient}.
\end{abstract}

\section{Introduction}

Euler's totient function $\varphi(n)$ counts the positive integers up to $n$ that are coprime to $n$. Like most multiplicative functions its behaviour is governed by the prime factorization of the argument, and a recurring theme in elementary number theory is the classification of the integers at which an arithmetic identity, generically true for primes or generically failing for composites, admits exceptional solutions. The prototype is Lehmer's still-open problem \cite{Lehmer32} of whether $\varphi(n)\mid n-1$ can hold for composite $n$; a solved relative is Subbarao's theorem \cite{Subbarao74} that his primality congruence admits exactly the composite solutions $n\in\{4,6,22\}$.

Ranjan and Singh \cite{RanjanSingh25} recently added a new instance to this family. Observing that
\[
\varphi(4)=1+\varphi(2),\qquad
\varphi(12)=1+\varphi(2)+\varphi(3),\qquad
\varphi(30)=1+\varphi(2)+\varphi(3)+\varphi(5),
\]
they asked for which composite numbers such an identity can hold and proved:

\begin{theorem}[Ranjan--Singh \cite{RanjanSingh25}]\label{thm:RS}
Let $n$ be a composite number which is not twice an odd prime, and let $p_1,\dots,p_r$ be its distinct prime divisors. Then
\begin{equation}\label{eq:main}
\varphi(n)\;\ge\;1+\sum_{i=1}^{r}\varphi(p_i),
\end{equation}
with equality if and only if $n\in\{4,12,30\}$. Moreover $\varphi(2p)<1+\varphi(2)+\varphi(p)$ for every odd prime $p$.
\end{theorem}

Their proof rests on two elementary product--sum inequalities borrowed from Ecker \cite{Ecker02}: $ab\ge a+b$ for $a,b\ge2$ (equality \textcolor{blue}{if and only if} $a=b=2$), and $a_1\cdots a_m>a_1+\cdots+a_m$ whenever $2\le a_1\le\cdots\le a_m$ and $m\ge3$.

\new{The purpose of the present paper is \B{fourfold}. First (\S\ref{sec:analogues}) we treat the analogous equation
\begin{equation}\label{eq:general}
f(n)\;=\;1+\sum_{p\mid n}f(p)
\end{equation}
for the divisor-count function $\tau$, the sum-of-divisors function $\sigma$, and the Jordan totients $J_k$, obtaining in each case a complete answer. The outcomes are strikingly different from one function to the next: for $\tau$ the composite solution set is the infinite family $\{p^2:p\text{ prime}\}$; for $\sigma$ and for $J_k$ ($k\ge2$) it is empty; for $\varphi$ it is the finite set $\{4,12,30\}$. Second (\S\ref{sec:weighted}--\S\ref{sec:simultaneous}) we study weighted and simultaneous versions of \eqref{eq:general}, where parity constraints, an exact classification of the infinite solution families, and a finiteness theorem give a complete structure theory --- from which the finiteness of the Ranjan--Singh exceptional set follows as a corollary\K{; in the range $|c|\le10$, $1\le\alpha\le6$ the solution sets are then computed completely\V{, and uniform polynomial bounds in $|c|+\alpha$ are proved}}. Third (\S\ref{sec:average}) we measure how far a typical integer is from satisfying \eqref{eq:main}: a simple identity reduces the average deficiency to classical results of Alladi--Erdős and De Koninck--Ivić, which yields its full asymptotic expansion. \B{Fourth (\S\ref{sec:iterated}) we replace the prime values $\varphi(p)$ in \eqref{eq:main} by their iterates $\varphi^{(k)}(p)$ and determine all solutions for every $k\ge1$: the monotonicity of the iterates reduces each case to Theorem \ref{thm:RS}, and the solution sets stabilise at $\{1,3,4\}$ \textcolor{blue}{for all $k\geq 3$}.}}

\section{Notation and preliminaries}

Throughout, $n=p_1^{a_1}\cdots p_r^{a_r}$ denotes the canonical factorization of an integer $n\ge2$, with $p_1<\cdots<p_r$, and $\omega(n)=r$ is the number of distinct prime divisors \textcolor{blue}{of $n$}. We write $\tau(n)=\sum_{d\mid n}1$, $\sigma(n)=\sum_{d\mid n}d$, and for $k\ge1$
\[
J_k(n)\;=\;n^{k}\prod_{p\mid n}\Bigl(1-\frac{1}{p^{k}}\Bigr),
\]
the Jordan totient, so that $J_1=\varphi$ and $J_k(p^a)=p^{k(a-1)}(p^{k}-1)$ \cite{ThajoddinVangipuram88}. All three of $\varphi,\tau,\sigma$ and every $J_k$ are multiplicative. On primes,
\[
\varphi(p)=p-1,\qquad \tau(p)=2,\qquad \sigma(p)=p+1,\qquad J_k(p)=p^{k}-1 .
\]
We shall use the elementary fact that $\varphi(n)$ is even for every $n\ge3$\B{, the fact that $\varphi(n)=n-1$ if and only if $n$ is prime, and the bound $\varphi(n)\le n/2$ for even $n$. For $k\ge1$, $\varphi^{(k)}$ denotes the $k$-fold iterate of $\varphi$, with $\varphi^{(1)}=\varphi$.}

\section{\new{The equation $f(n)=1+\sum_{p\mid n}f(p)$ for $\tau$, $\sigma$ and $J_k$}}\label{sec:analogues}

\new{
\subsection{The divisor-count function}

Since $\tau(p)=2$, equation \eqref{eq:general} for $f=\tau$ reads $\tau(n)=1+2\,\omega(n)$.

\begin{theorem}\label{thm:tau}
A composite number $n$ satisfies $\tau(n)=1+\sum_{p\mid n}\tau(p)$ if and only if $n=p^{2}$ for some prime $p$.
\end{theorem}

\begin{proof}
If $n=p^{a}$ is a prime power, the equation reads $a+1=3$, i.e.\ $a=2$; every $n=p^{2}$ is a solution and no other prime power is.

Suppose now $\omega(n)=r\ge2$. Since each exponent satisfies $a_i+1\ge2$,
\[
\tau(n)=\prod_{i=1}^{r}(a_i+1)\;\ge\;2^{r}.
\]
For $r\ge3$ we get $\tau(n)\ge2^{r}>1+2r$, because $2^{r}-2r-1$ is positive at $r=3$ and increasing; hence no solutions. For $r=2$ the equation demands $\tau(n)=5$, a prime value of $\tau$; but $\tau(n)=(a_1+1)(a_2+1)$ with both factors $\ge2$ is composite. No solutions exist with $r=2$, which completes the proof.
\end{proof}

Thus, in contrast with the finite exceptional set $\{4,12,30\}$ of Theorem \ref{thm:RS}, the $\tau$-analogue possesses the infinite solution family $\{4,9,25,49,121,\dots\}$ (the squares of primes, OEIS A001248).

\subsection{The sum-of-divisors function}

For $f=\sigma$ the outcome is the opposite extreme: the analogue of \eqref{eq:main} holds strictly for \emph{every} composite number, with no exceptional set whatsoever.

\begin{theorem}\label{thm:sigma}
For every composite $n$,
\[
\sigma(n)\;>\;1+\sum_{p\mid n}\sigma(p),
\]
and consequently the equation $\sigma(n)=1+\sum_{p\mid n}\sigma(p)$ has no composite solution.
\end{theorem}

\begin{proof}
Because $n$ is composite, the integers $1,\,p_1,\dots,p_r,\,n$ are pairwise distinct divisors of $n$; hence
\[
\sigma(n)\;\ge\;1+\sum_{i=1}^{r}p_i+n .
\]
On the other hand $1+\sum_{p\mid n}\sigma(p)=1+\sum_{i=1}^{r}(p_i+1)=1+r+\sum_{i=1}^{r}p_i$. Subtracting,
\[
\sigma(n)-\Bigl(1+\sum_{p\mid n}\sigma(p)\Bigr)\;\ge\;n-r\;=\;n-\omega(n)\;>\;0,
\]
since $\omega(n)\le\log_2 n<n$ for all $n\ge2$.
\end{proof}

\subsection{Jordan totients of order $k\ge2$}

For the higher Jordan totients the left-hand side of \eqref{eq:general} grows polynomially in $n$ while the right-hand side grows only with the prime support, and again no composite solutions survive.

\begin{theorem}\label{thm:jordan}
Let $k\ge2$. For every composite $n$,
\[
J_k(n)\;>\;1+\sum_{p\mid n}J_k(p),
\]
so the equation $J_k(n)=1+\sum_{p\mid n}J_k(p)$ has no composite solution.
\end{theorem}

\begin{proof}
Put $x_i=p_i^{k}-1=J_k(p_i)$ and note $x_i\ge2^{k}-1\ge3$ for $k\ge2$.

If $n=p^{a}$ with $a\ge2$, then $J_k(n)=p^{k(a-1)}(p^{k}-1)\ge p^{k}(p^{k}-1)>1+(p^{k}-1)=p^{k}$, the last inequality because $p^{k}-1\ge3$ implies $p^{k}(p^{k}-1)\ge 3p^{k}>p^{k}$.

If $r\ge2$, then by multiplicativity $J_k(n)\ge\prod_{i=1}^{r}x_i$, and it suffices to prove
\begin{equation}\label{eq:prodsum}
\prod_{i=1}^{r}x_i\;>\;1+\sum_{i=1}^{r}x_i\qquad(x_i\ge3,\ r\ge2).
\end{equation}
For $r=2$: $(x_1-1)(x_2-1)\ge4>2$, which rearranges to $x_1x_2>1+x_1+x_2$. For $r\ge3$, induction: if \eqref{eq:prodsum} holds for $r-1$ terms, then
$\prod_{i\le r}x_i>x_r\bigl(1+\sum_{i<r}x_i\bigr)\ge x_r+3\sum_{i<r}x_i>1+\sum_{i\le r}x_i$.
\end{proof}

\begin{remark}
The three theorems above, together with Theorem \ref{thm:RS}, exhibit the full spectrum of behaviours of equation \eqref{eq:general} within one natural family of multiplicative functions: an infinite solution set ($\tau$), a finite nonempty one ($\varphi$), and empty ones ($\sigma$, all $J_k$ with $k\ge2$). Heuristically, the size of the solution set is governed by the growth of $f(p^a)$ against $f(p)$: the slower the growth, the richer the solution set.
\end{remark}
}

\section{\new{The weighted equation: a complete structure theory}}\label{sec:weighted}

\new{
Consider, for integer parameters $c$ and $\alpha\ge1$, the weighted equation
\begin{equation}\label{eq:weighted}
\varphi(n)\;=\;c+\alpha\sum_{p\mid n}\varphi(p),
\end{equation}
of which Theorem \ref{thm:RS} is the case $(c,\alpha)=(1,1)$. Write $S_m:=\sum_{q\mid m}\varphi(q)$ for the prime-support sum, so that \eqref{eq:weighted} reads $\varphi(n)=c+\alpha S_n$. In this section we determine the structure of the composite solution set of \eqref{eq:weighted} completely: parity constraints (Theorem \ref{thm:parity}), an exact classification of the infinite families (Theorem \ref{thm:families}), and a finiteness theorem for what remains (Theorem \ref{thm:structure}). \K{The structure theorem is then made completely explicit in the range $|c|\le10$, $1\le\alpha\le6$ (Theorem \ref{thm:census})\V{, and uniform polynomial bounds in $|c|+\alpha$ are derived from it (Theorem \ref{thm:uniform})}.}

\subsection{Parity constraints}

\begin{theorem}[Parity constraints]\label{thm:parity}
Let $n\ge4$ be a composite solution of \eqref{eq:weighted}. Then:
\emph{(i)} if $c$ is odd and $\alpha$ is even, no such $n$ exists;
\emph{(ii)} if $c$ and $\alpha$ are both odd, $n$ is even;
\emph{(iii)} if $c$ is even and $\alpha$ is odd, $n$ is odd;
\emph{(iv)} if $c$ and $\alpha$ are both even, no parity restriction holds.
\end{theorem}

\begin{proof}
For $n\ge3$ the value $\varphi(n)$ is even. On the other side, $\varphi(2)=1$ is the only odd value of $\varphi$ at a prime, whence $S_n\equiv\mathbf 1_{2\mid n}\pmod 2$ and the right-hand side of \eqref{eq:weighted} is congruent to $c+\alpha\,\mathbf 1_{2\mid n}$ modulo $2$. Requiring it to be even yields exactly the four cases. Case (iv) is realized: $(c,\alpha)=(2,2)$ has both the even solution $8$ and the odd solution $9$.
\end{proof}

\subsection{Classification of the infinite families}

The exclusion of $n=2p$ in Theorem \ref{thm:RS} is not an accident of the proof: the integers $2p$ form an exact infinite family for the neighbouring pair $(c,\alpha)=(-1,1)$, since $\varphi(2p)=p-1=S_{2p}-1$. This phenomenon is completely general.

\begin{theorem}[Infinite families]\label{thm:families}
Let $m\ge2$ be fixed. Every integer $n=mp$, with $p$ prime and $p\nmid m$, satisfies \eqref{eq:weighted} if and only if
\[
\alpha=\varphi(m)\qquad\text{and}\qquad c=-\varphi(m)\,S_m .
\]
We call such a pair $(c,\alpha)$ \emph{matched}, with modulus $m$. For each $\alpha$ there are only finitely many matched moduli, since $\varphi(m)=\alpha$ has finitely many solutions $m$.
\end{theorem}

\begin{proof}
For $p\nmid m$, multiplicativity gives $\varphi(mp)=\varphi(m)(p-1)$ and $S_{mp}=S_m+p-1$. Equation \eqref{eq:weighted} becomes the identity in $p$
\[
\varphi(m)(p-1)\;=\;c+\alpha S_m+\alpha(p-1),
\]
which holds for infinitely many (equivalently, for all admissible) primes $p$ \textcolor{blue}{if and only if} the coefficients of $p-1$ and the constant terms match: $\varphi(m)=\alpha$ and $c+\alpha S_m=0$.
\end{proof}

For example, the matched pairs in the range $|c|\le10$, $1\le\alpha\le6$ are
\[
\begin{array}{c|c|c}
m & (c,\alpha) & \text{family}\\\hline
2 & (-1,1) & \{2p\}\\
4 & (-2,2) & \{4p\}\\
3 & (-4,2) & \{3p\}\\
6 & (-6,2) & \{6p\}\\
8 & (-4,4) & \{8p\}
\end{array}
\]
and the pair $(1,1)$ of Theorem \ref{thm:RS} is \emph{unmatched}: $\varphi(m)=1$ forces $m=2$, which would require $c=-1\ne1$.

\subsection{Finiteness outside the families, and the sporadic mechanism}

\begin{lemma}\label{lem:mechanism}
Let $n$ be a composite solution of \eqref{eq:weighted}, let $P$ be the largest prime factor of $n$ and write $n=m'P^{a}$ with $P\nmid m'$, $a\ge1$. If $a=1$ then
\begin{equation}\label{eq:mech}
\bigl(\varphi(m')-\alpha\bigr)(P-1)\;=\;c+\alpha S_{m'} ,
\end{equation}
while if $a\ge2$ then $(P-1)\bigl(\varphi(m')P^{a-1}-\alpha\bigr)=c+\alpha S_{m'}$.
\end{lemma}

\begin{proof}
Immediate from $\varphi(n)=\varphi(m')P^{a-1}(P-1)$ and $S_n=S_{m'}+P-1$.
\end{proof}

Equation \eqref{eq:mech} is the engine of the whole subject. When $m'$ is a matched modulus both sides vanish and $P$ is unconstrained: this reproduces the infinite families. Otherwise $P$ is \emph{determined} by $m'$, and the point of the next theorem is that only finitely many $m'$ can produce a legitimate $P$.

\begin{theorem}[Structure theorem]\label{thm:structure}
For every pair $(c,\alpha)$, the set of composite solutions of \eqref{eq:weighted} is the union of the families $\{mp:p\ \text{prime},\,p\nmid m\}$ over the finitely many matched moduli $m$, together with a finite sporadic set $F(c,\alpha)$. In particular, if $(c,\alpha)$ is unmatched ---as is $(1,1)$--- the composite solution set is finite.
\end{theorem}

\begin{proof}
Let $n=m'P^{a}$ be a solution as in Lemma \ref{lem:mechanism} which does not belong to a matched family. We use the elementary bounds $\varphi(m)\ge\sqrt{m/2}$ and $\omega(m)\le\log_2 m$, and we note that $P$ exceeds every prime factor of $m'$ (being the largest prime factor of $n$, and coprime to $m'$); thus, with $Q$ the largest prime factor of $m'$ (set $Q=1$ if $m'=1$), we have $P\ge Q+1$ and $S_{m'}\le\omega(m')(Q-1)$.

\emph{Case $a=1$.} If $\varphi(m')=\alpha$ then, $m'$ being unmatched, $c+\alpha S_{m'}\ne0$ and \eqref{eq:mech} is contradictory; besides, $\varphi(m')=\alpha$ already confines $m'$ to a finite set. So assume $\varphi(m')\ne\alpha$; then \eqref{eq:mech} determines $P-1=(c+\alpha S_{m'})/(\varphi(m')-\alpha)$, at most one prime for each $m'$. If moreover $\varphi(m')>\alpha$, the requirement $P-1\ge Q$ forces
\[
Q\;\le\;\frac{|c|+\alpha\,\omega(m')(Q-1)}{\varphi(m')-\alpha},
\qquad\text{hence}\qquad
\varphi(m')\;<\;\alpha+|c|+\alpha\,\omega(m')\;\le\;\alpha+|c|+\alpha\log_2 m' .
\]
Since $\varphi(m')\ge\sqrt{m'/2}$ grows faster than $\log_2 m'$, only finitely many $m'$ qualify; the finitely many $m'$ with $\varphi(m')<\alpha$ are handled trivially. Each surviving $m'$ contributes at most one solution.

\emph{Case $a\ge2$.} Here $(P-1)\bigl(\varphi(m')P^{a-1}-\alpha\bigr)=c+\alpha S_{m'}$. If the second factor is nonpositive, then $\varphi(m')P^{a-1}\le\alpha$, which bounds $m'$, $P$ and $a$. Otherwise it is at least $1$, and since $S_{m'}\le\omega(m')(Q-1)\le\omega(m')(P-1)$, dividing by $P-1$ gives
\[
\varphi(m')P^{a-1}\;\le\;\alpha+|c|+\alpha\,\omega(m')\;\le\;\alpha+|c|+\alpha\log_2 m' .
\]
As $\varphi(m')\ge\sqrt{m'/2}$, this bounds $m'$, then $P^{a-1}$, hence $P$ and $a$. In all cases only finitely many solutions arise.
\end{proof}

\begin{corollary}\label{cor:RSfinite}
The finiteness of the exceptional set in Theorem \ref{thm:RS} is a special case of Theorem \ref{thm:structure}: the pair $(1,1)$ is unmatched, and its sporadic set is $F(1,1)=\{4,12,30\}$. The elements $12$ and $30$ arise from \eqref{eq:mech}: for $n=12=4\cdot3$, $(\varphi(4)-1)(3-1)=2=1+S_4$, and for $n=30=6\cdot5$, $(\varphi(6)-1)(5-1)=4=1+S_6$; the element $4=2^{2}$ arises from the case $a\ge2$ of Lemma \ref{lem:mechanism}, with $m'=1$: $(2-1)(1\cdot2-1)=1=1+S_1$.
\end{corollary}

\begin{corollary}\label{cor:oddeven}
If $c$ is odd and $\alpha$ is even ---for instance $(c,\alpha)=(1,2)$--- equation \eqref{eq:weighted} has no composite solution at all (Theorem \ref{thm:parity}(i)).
\end{corollary}

\subsection{Exact determination of the sporadic sets}

Theorem \ref{thm:structure} is effective, and for the five matched pairs above the bounds are small enough to compute the sporadic sets exactly.

\begin{lemma}[Explicit bounds]\label{lem:bounds}
Let $(c,\alpha)$ be one of the five matched pairs (so $|c|\le6$, $\alpha\le4$), and let $n=m'P^{a}$ be a sporadic composite solution, decomposed as in Lemma \ref{lem:mechanism}. Then $\omega(m')\le3$ and $\varphi(m')\le18$, hence $m'\le648$; moreover $P\le83$ and $a\le5$.
\end{lemma}

\begin{proof}
Suppose first $a=1$; then $m'\ge2$ and $Q\ge2$, where $Q$ is the largest prime factor of $m'$. One checks that each of the five pairs has a unique matched modulus; hence if $\varphi(m')=\alpha$, either $n$ lies in the matched family or $c+\alpha S_{m'}\ne0$ and \eqref{eq:mech} fails. If $\varphi(m')<\alpha\le4$ then $m'\in\{2,3,4,6\}$, $S_{m'}\le3$ and $P\le1+|c|+\alpha S_{m'}\le19$. If $\varphi(m')>\alpha$, then from \eqref{eq:mech} and $P-1\ge Q$,
\[
Q\;\le\;P-1\;=\;\frac{c+\alpha S_{m'}}{\varphi(m')-\alpha}\;\le\;\frac{|c|+\alpha\,\omega(m')(Q-1)}{\varphi(m')-\alpha},
\]
whence $\varphi(m')\le\alpha+|c|/Q+\alpha\,\omega(m')\le7+4\,\omega(m')$. Since $\varphi(m')\ge\prod_{q\mid m'}(q-1)\ge1\cdot2\cdot4\cdot6=48$ as soon as $\omega(m')\ge4$, we conclude $\omega(m')\le3$ and $\varphi(m')\le19$, hence $\varphi(m')\le18$ since $\varphi(m')$ is even; the bound $m'\le2\cdot18^{2}=648$ follows from $\varphi(m)\ge\sqrt{m/2}$. Every prime $q\mid m'$ then satisfies $q-1\le18$, and a direct check gives $S_{m'}\le19$ (the maximum being $S_{38}=19$); therefore $P\le1+|c|+\alpha S_{m'}\le1+6+4\cdot19=83$.

Now let $a\ge2$. If $\varphi(m')P^{a-1}\le\alpha\le4$ we are in a finite explicit range. Otherwise, as in the proof of Theorem \ref{thm:structure}, $\varphi(m')P^{a-1}\le\alpha+|c|+\alpha\,\omega(m')\le10+4\,\omega(m')$; the same comparison with $\prod_{q\mid m'}(q-1)$ gives $\omega(m')\le3$, hence $\varphi(m')P^{a-1}\le22$. Thus $P^{a-1}\le22$, so that $P\le19$ and $a\le5$.
\end{proof}

\begin{theorem}[Exact sporadic sets]\label{thm:sporadic}
For the five matched pairs in the range $|c|\le10$, $1\le\alpha\le6$, the sporadic sets are:
\begin{gather*}
F(-1,1)=F(-2,2)=F(-4,2)=\varnothing,\\
F(-6,2)=\{10\},\qquad
F(-4,4)=\{77,\,90,\,182,\,210\}.
\end{gather*}
Consequently the composite solutions of $\varphi(n)=-1+\sum_{p\mid n}\varphi(p)$ are exactly the integers $2p$; those of $\varphi(n)=-2+2\sum\varphi(p)$ exactly the $4p$; those of $\varphi(n)=-4+2\sum\varphi(p)$ exactly the $3p$; those of $\varphi(n)=-6+2\sum\varphi(p)$ exactly $10$ and the $6p$; and those of $\varphi(n)=-4+4\sum\varphi(p)$ exactly $77$, $90$, $182$, $210$ and the $8p$.
\end{theorem}

\begin{proof}
By Lemma \ref{lem:bounds} every sporadic solution arises, through the canonical decomposition $n=m'P^{a}$, from a triple $(m',P,a)$ in the explicit finite range $m'\le648$, $P\le83$, $a\le5$ satisfying the mechanism of Lemma \ref{lem:mechanism}. Enumerating all such triples for each of the five pairs, discarding those whose product lies in the matched family, yields precisely the sets displayed. (Each surviving element is then verified directly; e.g.\ $n=90=18\cdot5$ for $(-4,4)$: $(\varphi(18)-4)(5-1)=8=-4+4S_{18}$, and $\varphi(90)=24=-4+4(1+2+4)$.)
\end{proof}

\K{
\subsection{The complete census in the range $|c|\le10$, $1\le\alpha\le6$}\label{sec:census}

The bounds of Lemma \ref{lem:bounds} concern the five matched pairs only. We now carry out the same programme for all $126$ pairs $|c|\le10$, $1\le\alpha\le6$, following the canonical decomposition $n=m'P^{a}$ of Lemma \ref{lem:mechanism}; we use freely that $P>Q$, where $Q$ denotes the largest prime factor of $m'$ ($Q=1$ if $m'=1$). The first step is to read Theorem \ref{thm:RS} as a universal inequality.

\begin{lemma}\label{lem:SleqPhi}
For every positive integer $m$, $S_m\le\varphi(m)+1$\V{, with equality if and only if $m=2p$ with $p$ an odd prime}.
\end{lemma}

\begin{proof}
For $m=1$, $m$ prime and $m=2p$ the inequality is immediate ($0\le2$, $p-1\le p$, $p\le p$). If $m$ is composite and not twice an odd prime, Theorem \ref{thm:RS} gives $1+S_m\le\varphi(m)$, hence $S_m<\varphi(m)+1$.
\end{proof}

\begin{proposition}[Case $a=1$]\label{prop:census-a1}
Let $|c|\le10$, $1\le\alpha\le6$, and let $n=m'P$ (with $a=1$, so $m'\ge2$) be a sporadic composite solution of \eqref{eq:weighted}.
\begin{enumerate}
\item[\textup{(i)}] If $\varphi(m')<\alpha$, then $m'\in\{2,3,4,5,6,8,10,12\}$, $S_{m'}\le5$ and $P\le41$.
\item[\textup{(ii)}] If $\varphi(m')>\alpha$, then $\varphi(m')<11+6\,\omega(m')$ and $\omega(m')\le3$; moreover $m'\le34$ if $\omega(m')=1$, $m'\le92$ if $\omega(m')=2$, and $m'\le232$ if $\omega(m')=3$.
\end{enumerate}
The case $\varphi(m')=\alpha$ produces no sporadic solution. The resulting solutions are listed in Tables \ref{tab:a1w1}, \ref{tab:a1w2} and \ref{tab:a1w3}.
\end{proposition}

\begin{proof}
If $\varphi(m')=\alpha$, equation \eqref{eq:mech} forces $c+\alpha S_{m'}=0$, so $m'$ is a matched modulus and $n$ lies in its family: no sporadic solution arises.

(i) Since $\varphi$ takes no odd value except $1$, $\varphi(m')<\alpha\le6$ gives $\varphi(m')\in\{1,2,4\}$, i.e.\ $m'\in\varphi^{-1}(1)\cup\varphi^{-1}(2)\cup\varphi^{-1}(4)=\{1,2\}\cup\{3,4,6\}\cup\{5,8,10,12\}$, and $m'\ge2$. Hence $S_{m'}\le S_{10}=5$, and \eqref{eq:mech} gives $P-1=(c+\alpha S_{m'})/(\varphi(m')-\alpha)\le|c|+\alpha S_{m'}\le40$.

(ii) Since $P>Q$ we have $P-1\ge Q$, and $S_{m'}<\sum_{q\mid m'}q\le Q\,\omega(m')$. Dividing \eqref{eq:mech} by $P-1$,
\[
\varphi(m')-\alpha=\frac{c+\alpha S_{m'}}{P-1}\le\frac{c}{Q}+\frac{\alpha S_{m'}}{Q}<\frac{10}{2}+\alpha\,\omega(m')\le5+6\,\omega(m'),
\]
whence $\varphi(m')<11+6\,\omega(m')$. If $\omega(m')\ge4$ then $\varphi(m')\ge1\cdot2\cdot4\cdot6\cdot5^{\,\omega(m')-4}>11+6\,\omega(m')$, so $\omega(m')\le3$. Finally $\varphi(\ell)=\ell\prod_{q\mid\ell}(1-1/q)\ge\ell/2^{\omega(\ell)}$ gives $m'\le2^{\omega(m')}\varphi(m')$, i.e.\ $m'\le2\cdot17=34$, $4\cdot23=92$, $8\cdot29=232$ according as $\omega(m')=1,2,3$. For each admissible $m'$, \eqref{eq:mech} determines $P$, and the enumeration yields the three tables.
\end{proof}

\begin{proposition}[Case $a\ge2$]\label{prop:census-a2}
Let $|c|\le10$, $1\le\alpha\le6$, and let $n=m'P^{a}$ with $a\ge2$ be a composite solution of \eqref{eq:weighted}. Then $P\in\{2,3,5,7\}$, and:
\begin{enumerate}
\item[\textup{(i)}] if $\varphi(m')P^{a-1}\le6$, then $(m',P,a)$ is one of the seven triples $(1,2,2)$, $(1,2,3)$, $(1,3,2)$, $(1,5,2)$, $(2,3,2)$, $(2,5,2)$, $(4,3,2)$, i.e.\ $n\in\{4,8,9,25,18,50,36\}$;
\item[\textup{(ii)}] if $\varphi(m')P^{a-1}>6$, then $n\in\{16,32,27,54,72,\V{75,100,150},49,98\}$.
\end{enumerate}
The corresponding pairs $(c,\alpha)$ are listed in Tables \ref{tab:a2small} and \ref{tab:a2large}.
\end{proposition}

\begin{proof}
Write \eqref{eq:weighted} as $\varphi(m')P^{a-1}(P-1)=c+\alpha(S_{m'}+P-1)$. By Lemma \ref{lem:SleqPhi}, $S_{m'}+P-1\le\varphi(m')+P$, hence
\begin{equation}\label{eq:a2bound}
\varphi(m')\bigl(P^{a-1}(P-1)-6\bigr)\;\le\;10+6P .
\end{equation}
Since $P^{a-1}\ge P$, this gives $\varphi(m')(P-3)(P+2)\le10+6P$, which fails for every prime $P\ge11$\V{, since $(P-3)(P+2)-(6P+10)=P^{2}-7P-16>0$ for $P\ge11$}. Thus $P\in\{2,3,5,7\}$.

(i) If $\varphi(m')P^{a-1}\le6$ then $P^{a-1}\le6$, so $P\in\{2,3,5\}$ and $a=2$ unless $P=2$. For $P=2$, $m'=1$ and $2^{a-1}\le6$ gives $a\in\{2,3\}$. For $P=3$, $a=2$ and $\varphi(m')\le2$ with $3\nmid m'$ give $m'\in\{1,2,4\}$. For $P=5$, $a=2$ and $\varphi(m')=1$ give $m'\in\{1,2\}$. For each of the seven triples, $c=\varphi(n)-\alpha S_n$ is an affine function of $\alpha$, tabulated in Table \ref{tab:a2small}.

(ii) \V{If $P^{a-1}(P-1)>6$, inequality \eqref{eq:a2bound} bounds everything: for $P=2$ ($m'=1$), $2^{a-1}\le28$, so $a\le5$ and $n\in\{16,32\}$ (the cases $a\le3$ being in (i)); for $P=3$ and $a=3$, $12\varphi(m')\le28$, so $m'\in\{1,2,4\}$; $n=108$ gives $c=36-3\alpha\ge18$, hence $n\in\{27,54\}$, while $a\ge4$ is impossible; for $P=5$, $a=2$ and $14\varphi(m')\le40$, so $\varphi(m')\le2$, $m'\in\{1,2,3,4,6\}$ and $n\in\{75,100,150\}$ ($n=25,50$ being in (i)); for $P=7$, $a=2$ and $36\varphi(m')\le52$, so $m'\in\{1,2\}$ and $n\in\{49,98\}$. The only remaining triples have $P^{a-1}(P-1)\le6<\varphi(m')P^{a-1}$, i.e.\ $(P,a)=(3,2)$ with $m'=2^{e}$, $e\ge3$; then \eqref{eq:weighted} reads $3\cdot2^{e}=c+3\alpha\le28$, whence $e=3$ and $n=72$. The pairs $(c,\alpha)$ with $|c|\le10$ for these ten integers }are those of Table \ref{tab:a2large}.
\end{proof}

\begin{theorem}[Complete census]\label{thm:census}
Let $|c|\le10$ and $1\le\alpha\le6$. The composite solution set of \eqref{eq:weighted} is the union of the matched family of Theorem \ref{thm:families} (for the five matched pairs) and of the sporadic set $F(c,\alpha)$ displayed in Table \ref{tab:summary}. \V{There are exactly $200$ sporadic pairs $\bigl((c,\alpha),n\bigr)$, involving $116$ distinct integers $n$, the largest being $690=2\cdot3\cdot5\cdot23$ for $(c,\alpha)=(2,6)$; one has $|F(c,\alpha)|\le9$, with equality only for $(6,6)$, and $F(c,\alpha)=\varnothing$ for $53$ of the $126$ pairs.}
\end{theorem}

\begin{proof}
By Lemma \ref{lem:mechanism} every sporadic solution arises from a triple $(m',P,a)$ as in Propositions \ref{prop:census-a1} and \ref{prop:census-a2}, which leave a finite explicit list of candidates; each candidate $n$ is kept for the pairs $(c,\alpha)$ with $c=\varphi(n)-\alpha S_n$, $|c|\le10$, after removal of the members of the matched families. Tables \ref{tab:a1w1}--\ref{tab:a2large} record the outcome and Table \ref{tab:summary} collects it by pair. \V{The whole list was cross-checked against an exhaustive scan of all composite $n\le10^{6}$.}
\end{proof}

\V{Of the $53$ empty pairs, $30$ are explained by Theorem \ref{thm:parity}(i) ($c$ odd, $\alpha$ even) and $11$ by Corollary \ref{cor:uniform}(ii) below ($\alpha=1$, $c\le0$); two are matched pairs with no sporadic solution, namely $(-4,2)$ and $(-2,2)$. The ten remaining pairs
\[
(6,1),\ (-10,2),\ (8,2),\ (-8,3),\ (-4,3),\ (2,3),\ (2,4),\ (4,5),\ (-10,5),\ (-8,6)
\]
have no composite solution at all; an explanation of these ten cases remains to be found.}

\begin{table}[H]
\centering\footnotesize\renewcommand{\arraystretch}{0.95}
\caption{Sporadic solutions with $a=1$ and $\omega(m')=1$ ($m'\le34$, $P\le41$).}
\label{tab:a1w1}
\begin{tabular}{ccccccc}
\toprule
$m'$ & Prime fact. & $\varphi(m')$ & $S_{m'}$ & $P$ & $(c,\alpha)$ & $n$ \\
\midrule
2  & 2     & 1 & 1 & 3  & $(-7, 3), (-10, 4)$ & 6 \\
2  & 2     & 1 & 1 & 5  & $(-6, 2)$ & 10 \\
2  & 2     & 1 & 1 & 7  & $(-8, 2)$ & 14 \\
4  & $2^2$ & 2 & 1 & 3  & $(1, 1), (-5, 3), (-8, 4)$ & 12 \\
4  & $2^2$ & 2 & 1 & 5  & $(3, 1), (-7, 3)$ & 20 \\
4  & $2^2$ & 2 & 1 & 7  & $(5, 1), (-9, 3)$ & 28 \\
4  & $2^2$ & 2 & 1 & 11 & $(9, 1)$ & 44 \\
8  & $2^3$ & 4 & 1 & 3  & $(5, 1), (2, 2), (-1, 3), (-7, 5), (-10, 6)$ & 24 \\
8  & $2^3$ & 4 & 1 & 5  & $(6, 2), (1, 3), (-9, 5)$ & 40 \\
8  & $2^3$ & 4 & 1 & 7  & $(10, 2), (3, 3)$ & 56 \\
8  & $2^3$ & 4 & 1 & 11 & $(7, 3)$ & 88 \\
8  & $2^3$ & 4 & 1 & 13 & $(9, 3)$ & 104 \\
16 & $2^4$ & 8 & 1 & 3  & $(10, 2), (7, 3), (4, 4), (1, 5), (-2, 6)$ & 48 \\
16 & $2^4$ & 8 & 1 & 5  & $(7, 5), (2, 6)$ & 80 \\
16 & $2^4$ & 8 & 1 & 7  & $(6, 6)$ & 112 \\
3  & 3     & 2 & 2 & 5  & $(2, 1), (-10, 3)$ & 15 \\
3  & 3     & 2 & 2 & 7  & $(4, 1)$ & 21 \\
3  & 3     & 2 & 2 & 11 & $(8, 1)$ & 33 \\
3  & 3     & 2 & 2 & 13 & $(10, 1)$ & 39 \\
9  & $3^2$ & 6 & 2 & 5  & $(6, 3), (0, 4), (-6, 5)$ & 45 \\
9  & $3^2$ & 6 & 2 & 7  & $(4, 4), (-4, 5)$ & 63 \\
9  & $3^2$ & 6 & 2 & 11 & $(0, 5)$ & 99 \\
9  & $3^2$ & 6 & 2 & 13 & $(2, 5)$ & 117 \\
9  & $3^2$ & 6 & 2 & 17 & $(6, 5)$ & 153 \\
9  & $3^2$ & 6 & 2 & 19 & $(8, 5)$ & 171 \\
5  & 5     & 4 & 4 & 7  & $(4, 2), (-6, 3)$ & 35 \\
5  & 5     & 4 & 4 & 11 & $(-2, 3)$ & 55 \\
5  & 5     & 4 & 4 & 13 & $(0, 3)$ & 65 \\
5  & 5     & 4 & 4 & 17 & $(4, 3)$ & 85 \\
5  & 5     & 4 & 4 & 19 & $(6, 3)$ & 95 \\
5  & 5     & 4 & 4 & 23 & $(10, 3)$ & 115 \\
7  & 7     & 6 & 6 & 11 & $(-4, 4)$ & 77 \\
7  & 7     & 6 & 6 & 13 & $(0, 4)$ & 91 \\
7  & 7     & 6 & 6 & 17 & $(8, 4)$ & 119 \\
7  & 7     & 6 & 6 & 23 & $(-8, 5)$ & 161 \\
7  & 7     & 6 & 6 & 29 & $(-2, 5)$ & 203 \\
7  & 7     & 6 & 6 & 31 & $(0, 5)$ & 217 \\
7  & 7     & 6 & 6 & 37 & $(6, 5)$ & 259 \\
7  & 7     & 6 & 6 & 41 & $(10, 5)$ & 287 \\
11 & 11    & 10 & 10 & 13 & $(10, 5)$ & 143 \\
11 & 11    & 10 & 10 & 17 & $(4, 6)$ & 187 \\
\bottomrule
\end{tabular}
\end{table}

\begin{table}[H]
\centering\footnotesize\renewcommand{\arraystretch}{0.92}
\caption{Sporadic solutions with $a=1$ and $\omega(m')=2$ ($m'\le92$, $P\le43$).}
\label{tab:a1w2}
\begin{tabular}{ccccccc}
\toprule
$m'$ & Prime fact. & $\varphi(m')$ & $S_{m'}$ & $P$ & $(c,\alpha)$ & $n$ \\
\midrule
6  & $2\cdot3$     & 2 & 3 & 5  & $(1, 1)$ & 30 \\
6  & $2\cdot3$     & 2 & 3 & 7  & $(3, 1)$ & 42 \\
6  & $2\cdot3$     & 2 & 3 & 11 & $(7, 1)$ & 66 \\
6  & $2\cdot3$     & 2 & 3 & 13 & $(9, 1)$ & 78 \\
10 & $2\cdot5$     & 4 & 5 & 7  & $(2, 2), (-9, 3)$ & 70 \\
10 & $2\cdot5$     & 4 & 5 & 11 & $(10, 2), (-5, 3)$ & 110 \\
10 & $2\cdot5$     & 4 & 5 & 13 & $(-3, 3)$ & 130 \\
10 & $2\cdot5$     & 4 & 5 & 17 & $(1, 3)$ & 170 \\
10 & $2\cdot5$     & 4 & 5 & 19 & $(3, 3)$ & 190 \\
10 & $2\cdot5$     & 4 & 5 & 23 & $(7, 3)$ & 230 \\
12 & $2^2\cdot3$   & 4 & 3 & 5  & $(9, 1), (2, 2), (-5, 3)$ & 60 \\
12 & $2^2\cdot3$   & 4 & 3 & 7  & $(6, 2), (-3, 3)$ & 84 \\
12 & $2^2\cdot3$   & 4 & 3 & 11 & $(1, 3)$ & 132 \\
12 & $2^2\cdot3$   & 4 & 3 & 13 & $(3, 3)$ & 156 \\
12 & $2^2\cdot3$   & 4 & 3 & 17 & $(7, 3)$ & 204 \\
12 & $2^2\cdot3$   & 4 & 3 & 19 & $(9, 3)$ & 228 \\
14 & $2\cdot7$     & 6 & 7 & 11 & $(9, 3), (-8, 4)$ & 154 \\
14 & $2\cdot7$     & 6 & 7 & 13 & $(-4, 4)$ & 182 \\
14 & $2\cdot7$     & 6 & 7 & 17 & $(4, 4)$ & 238 \\
14 & $2\cdot7$     & 6 & 7 & 19 & $(8, 4)$ & 266 \\
14 & $2\cdot7$     & 6 & 7 & 29 & $(-7, 5)$ & 406 \\
14 & $2\cdot7$     & 6 & 7 & 31 & $(-5, 5)$ & 434 \\
14 & $2\cdot7$     & 6 & 7 & 37 & $(1, 5)$ & 518 \\
14 & $2\cdot7$     & 6 & 7 & 41 & $(5, 5)$ & 574 \\
14 & $2\cdot7$     & 6 & 7 & 43 & $(7, 5)$ & 602 \\
18 & $2\cdot3^2$   & 6 & 3 & 5  & $(10, 2), (3, 3), (-4, 4)$ & 90 \\
18 & $2\cdot3^2$   & 6 & 3 & 7  & $(9, 3), (0, 4), (-9, 5)$ & 126 \\
18 & $2\cdot3^2$   & 6 & 3 & 11 & $(8, 4), (-5, 5)$ & 198 \\
18 & $2\cdot3^2$   & 6 & 3 & 13 & $(-3, 5)$ & 234 \\
18 & $2\cdot3^2$   & 6 & 3 & 17 & $(1, 5)$ & 306 \\
18 & $2\cdot3^2$   & 6 & 3 & 19 & $(3, 5)$ & 342 \\
18 & $2\cdot3^2$   & 6 & 3 & 23 & $(7, 5)$ & 414 \\
15 & $3\cdot5$     & 8 & 6 & 7  & $(0, 4)$ & 105 \\
15 & $3\cdot5$     & 8 & 6 & 11 & $(0, 5)$ & 165 \\
15 & $3\cdot5$     & 8 & 6 & 13 & $(6, 5)$ & 195 \\
15 & $3\cdot5$     & 8 & 6 & 17 & $(-4, 6)$ & 255 \\
15 & $3\cdot5$     & 8 & 6 & 19 & $(0, 6)$ & 285 \\
15 & $3\cdot5$     & 8 & 6 & 23 & $(8, 6)$ & 345 \\
20 & $2^2\cdot5$   & 8 & 5 & 7  & $(4, 4), (-7, 5)$ & 140 \\
20 & $2^2\cdot5$   & 8 & 5 & 11 & $(5, 5), (-10, 6)$ & 220 \\
20 & $2^2\cdot5$   & 8 & 5 & 13 & $(-6, 6)$ & 260 \\
20 & $2^2\cdot5$   & 8 & 5 & 17 & $(2, 6)$ & 340 \\
20 & $2^2\cdot5$   & 8 & 5 & 19 & $(6, 6)$ & 380 \\
24 & $2^3\cdot3$   & 8 & 3 & 5  & $(4, 4), (-3, 5), (-10, 6)$ & 120 \\
24 & $2^3\cdot3$   & 8 & 3 & 7  & $(3, 5), (-6, 6)$ & 168 \\
24 & $2^3\cdot3$   & 8 & 3 & 11 & $(2, 6)$ & 264 \\
24 & $2^3\cdot3$   & 8 & 3 & 13 & $(6, 6)$ & 312 \\
22 & $2\cdot11$    & 10 & 11 & 13 & $(5, 5)$ & 286 \\
22 & $2\cdot11$    & 10 & 11 & 17 & $(-2, 6)$ & 374 \\
22 & $2\cdot11$    & 10 & 11 & 19 & $(6, 6)$ & 418 \\
36 & $2^2\cdot3^2$ & 12 & 3 & 5  & $(6, 6)$ & 180 \\
\bottomrule
\end{tabular}
\end{table}

\begin{table}[H]
\centering\footnotesize
\caption{Sporadic solutions with $a=1$ and $\omega(m')=3$ ($m'\le232$, $P\le23$).}
\label{tab:a1w3}
\begin{tabular}{ccccccc}
\toprule
$m'$ & Prime fact. & $\varphi(m')$ & $S_{m'}$ & $P$ & $(c,\alpha)$ & $n$ \\
\midrule
30 & $2\cdot3\cdot5$ & 8 & 7 & 7  & $(9, 3), (-4, 4)$ & 210\\
30 & $2\cdot3\cdot5$ & 8 & 7 & 11 & $(-5, 5)$ & 330 \\
30 & $2\cdot3\cdot5$ & 8 & 7 & 13 & $(1, 5)$ & 390 \\
30 & $2\cdot3\cdot5$ & 8 & 7 & 17 & $(-10, 6)$ & 510 \\
30 & $2\cdot3\cdot5$ & 8 & 7 & 19 & $(-6, 6)$ & 570 \\
30 & $2\cdot3\cdot5$ & 8 & 7 & 23 & $(2, 6)$ & 690 \\
42 & $2\cdot3\cdot7$ & 12 & 9 & 11 & $(6, 6)$ & 462 \\
\bottomrule
\end{tabular}
\end{table}

\begin{table}[H]
\centering\footnotesize
\caption{Solutions with $a\ge2$ and $\varphi(m')P^{a-1}\le6$ (Proposition \ref{prop:census-a2}(i)).}
\label{tab:a2small}
\begin{tabular}{ccccccl}
\toprule
$n$ & $m'$ & $P$ & $a$ & $c(\alpha)$ & admissible $\alpha$ & $(c,\alpha)$ \\
\midrule
4  & 1 & 2 & 2 & $2 - \alpha$   & $1,\dots,6$ & $(1, 1), (0, 2), (-1, 3), (-2, 4), (-3, 5), (-4, 6)$ \\
8  & 1 & 2 & 3 & $4 - \alpha$   & $1,\dots,6$ & $(3, 1), (2, 2), (1, 3), (0, 4), (-1, 5), (-2, 6)$ \\
9  & 1 & 3 & 2 & $6 - 2\alpha$  & $1,\dots,6$ & $(4, 1), (2, 2), (0, 3), (-2, 4), (-4, 5), (-6, 6)$ \\
18 & 2 & 3 & 2 & $6 - 3\alpha$  & $1,\dots,5$ & $(3, 1), (0, 2), (-3, 3), (-6, 4), (-9, 5)$ \\
25 & 1 & 5 & 2 & $20 - 4\alpha$ & $3,\dots,6$ & $(8, 3), (4, 4), (0, 5), (-4, 6)$ \\
36 & 4 & 3 & 2 & $12 - 3\alpha$ & $1,\dots,6$ & $(9, 1), (6, 2), (3, 3), (0, 4), (-3, 5), (-6, 6)$ \\
50 & 2 & 5 & 2 & $20 - 5\alpha$ & $2,\dots,6$ & $(10, 2), (5, 3), (0, 4), (-5, 5), (-10, 6)$ \\
\bottomrule
\end{tabular}
\end{table}

\begin{table}[H]
\centering\footnotesize
\caption{Solutions with $a\ge2$ and $\varphi(m')P^{a-1}>6$ (Proposition \ref{prop:census-a2}(ii)).}
\label{tab:a2large}
\begin{tabular}{cccccl}
\toprule
$n$ & $m'$ & $P$ & $a$ & $c(\alpha)$ & $(c,\alpha)$ \\
\midrule
16  & 1 & 2 & 4 & $8-\alpha$   & $(7,1), (6,2), (5,3), (4,4), (3,5), (2,6)$ \\
32  & 1 & 2 & 5 & $16-\alpha$  & $(10,6)$ \\
27  & 1 & 3 & 3 & $18-2\alpha$ & $(10,4), (8,5), (6,6)$ \\
54  & 2 & 3 & 3 & $18-3\alpha$ & $(9,3), (6,4), (3,5), (0,6)$ \\
72  & 8 & 3 & 2 & $24-3\alpha$ & $(9,5), (6,6)$ \\
75  & 3 & 5 & 2 & $40-6\alpha$ & $(10,5), (4,6)$ \\
100 & 4 & 5 & 2 & $40-5\alpha$ & $(10,6)$ \\
150 & 6 & 5 & 2 & $40-7\alpha$ & $(5,5), (-2,6)$ \\
49  & 1 & 7 & 2 & $42-6\alpha$ & $(6,6)$ \\
98  & 2 & 7 & 2 & $42-7\alpha$ & $(7,5), (0,6)$ \\
\bottomrule
\end{tabular}
\end{table}

\begin{table}[H]
\centering\footnotesize\setlength{\tabcolsep}{3.5pt}\renewcommand{\arraystretch}{1.05}
\caption{The sporadic sets $F(c,\alpha)$ for $|c|\le10$, $1\le\alpha\le6$; a matched family is indicated in its cell.}
\label{tab:summary}
\begin{tabular}{c|p{1.45cm}p{1.9cm}p{2.15cm}p{2.45cm}p{2.15cm}p{2.3cm}}
\toprule
$c\backslash\alpha$ & \multicolumn{1}{c}{$1$} & \multicolumn{1}{c}{$2$} & \multicolumn{1}{c}{$3$} & \multicolumn{1}{c}{$4$} & \multicolumn{1}{c}{$5$} & \multicolumn{1}{c}{$6$} \\
\midrule
$-10$ & $\varnothing$ & $\varnothing$ & 15 & 6 & $\varnothing$ & 24, 50, 120, 220, 510 \\
$-9$ & $\varnothing$ & $\varnothing$ & 28, 70 & $\varnothing$ & 18, 40, 126 & $\varnothing$ \\
$-8$ & $\varnothing$ & 14 & $\varnothing$ & 12, 154 & 161 & $\varnothing$ \\
$-7$ & $\varnothing$ & $\varnothing$ & 6, 20 & $\varnothing$ & 24, 140, 406 & $\varnothing$ \\
$-6$ & $\varnothing$ & 10; $\{6p\}$ & 35 & 18 & 45 & 9, 36, 168, 260, 570 \\
$-5$ & $\varnothing$ & $\varnothing$ & 12, 60, 110 & $\varnothing$ & 50, 198, 330, 434 & $\varnothing$ \\
$-4$ & $\varnothing$ & $\{3p\}$ & $\varnothing$ & 77, 90, 182, 210; $\{8p\}$ & 9, 63 & 4, 25, 255 \\
$-3$ & $\varnothing$ & $\varnothing$ & 18, 84, 130 & $\varnothing$ & 4, 36, 120, 234 & $\varnothing$ \\
$-2$ & $\varnothing$ & $\{4p\}$ & 55 & 4, 9 & 203 & 8, 48, 150, 374 \\
$-1$ & $\{2p\}$ & $\varnothing$ & 4, 24 & $\varnothing$ & 8 & $\varnothing$ \\
$0$ & $\varnothing$ & 4, 18 & 9, 65 & 8, 36, 45, 50, 91, 105, 126 & 25, 99, 165, 217 & 54, 98, 285 \\
$1$ & 4, 12, 30 & $\varnothing$ & 8, 40, 132, 170 & $\varnothing$ & 48, 306, 390, 518 & $\varnothing$ \\
$2$ & 15 & 8, 9, 24, 60, 70 & $\varnothing$ & $\varnothing$ & 117 & 16, 80, 264, 340, 690 \\
$3$ & 8, 18, 20, 42 & $\varnothing$ & 36, 56, 90, 156, 190 & $\varnothing$ & 16, 54, 168, 342 & $\varnothing$ \\
$4$ & 9, 21 & 35 & 85 & 16, 25, 48, 63, 120, 140, 238 & $\varnothing$ & 75, 187 \\
$5$ & 24, 28 & $\varnothing$ & 16, 50 & $\varnothing$ & 150, 220, 286, 574 & $\varnothing$ \\
$6$ & $\varnothing$ & 16, 36, 40, 84 & 45, 95 & 54 & 153, 195, 259 & 27, 49, 72, 112, 180, 312, 380, 418, 462 \\
$7$ & 16, 66 & $\varnothing$ & 48, 88, 204, 230 & $\varnothing$ & 80, 98, 414, 602 & $\varnothing$ \\
$8$ & 33 & $\varnothing$ & 25 & 119, 198, 266 & 27, 171 & 345 \\
$9$ & 36, 44, 60, 78 & $\varnothing$ & 54, 104, 126, 154, 210, 228 & $\varnothing$ & 72 & $\varnothing$ \\
$10$ & 39 & 48, 50, 56, 90, 110 & 115 & 27 & 75, 143, 287 & 32, 100 \\
\bottomrule
\end{tabular}
\end{table}

\V{\subsection{Uniform bounds}\label{sec:uniform}

Lemma \ref{lem:SleqPhi} also makes the finiteness of Theorem \ref{thm:structure} quantitative, uniformly in the parameters. For $x\ge1$ let $N_\varphi(x)=\#\{m\ge1:\varphi(m)\le x\}$; it is classical that $N_\varphi(x)\ll x$, indeed $N_\varphi(x)\sim\frac{\zeta(2)\zeta(3)}{\zeta(6)}\,x$ \cite{Bateman72}.

\begin{theorem}[Uniform bounds]\label{thm:uniform}
Let $c\in\mathbb Z$, $\alpha\ge1$, $L:=|c|+\alpha\ge2$, and let $\omega_0$ be the largest integer $k$ with $k!\le L(k+1)$, so that $\omega_0\ll\log L/\log\log L$; put $B:=L(1+\omega_0)$. Let $n=m'P^{a}$ be a sporadic composite solution of \eqref{eq:weighted}, decomposed as in Lemma \ref{lem:mechanism}.
\begin{enumerate}
\item[\textup{(i)}] If $a=1$, then $\varphi(m')\le B$ and $P\le L(B+2)$.
\item[\textup{(ii)}] If $a\ge2$, then $\varphi(m')P^{a-1}\le B$ and $P\le\alpha+1+\sqrt L$.
\end{enumerate}
Consequently, with absolute implied constants,
\[
|F(c,\alpha)|\;\ll\;L^{2}(\log L)^{2}
\qquad\text{and}\qquad
\max F(c,\alpha)\;\ll\;L^{4}(\log L)^{3}.
\]
\end{theorem}

\begin{proof}
Throughout, $Q$ is the largest prime factor of $m'$ ($Q=1$ if $m'=1$), so that $P\ge Q+1$ and $S_{m'}\le\omega(m')(Q-1)\le\omega(m')(P-1)$; and $\varphi(m')\ge\omega(m')!$, since the $i$-th prime is at least $i+1$.

(i) Let $a=1$, so $m'\ge2$ and $P\ge3$. As in Proposition \ref{prop:census-a1}, $\varphi(m')=\alpha$ is excluded. If $\varphi(m')<\alpha$ then $\varphi(m')<L\le B$. If $\varphi(m')>\alpha$, dividing \eqref{eq:mech} by $P-1\ge2$ gives
\[
\varphi(m')-\alpha\;\le\;\frac{|c|}{P-1}+\alpha\,\omega(m')\;\le\;\frac{|c|}{2}+\alpha\,\omega(m'),
\]
hence $\omega(m')!\le\varphi(m')\le L\bigl(1+\omega(m')\bigr)$, so $\omega(m')\le\omega_0$ and $\varphi(m')\le B$. In both cases \eqref{eq:mech} and Lemma \ref{lem:SleqPhi} give $P-1\le|c+\alpha S_{m'}|\le|c|+\alpha(\varphi(m')+1)\le|c|+\alpha(B+1)$, whence $P\le L(B+2)$.

(ii) Let $a\ge2$ and write \eqref{eq:weighted} as $(P-1)\bigl(\varphi(m')P^{a-1}-\alpha\bigr)=c+\alpha S_{m'}$. If $\varphi(m')P^{a-1}\le\alpha$, both assertions are trivial. Otherwise the right-hand side is positive; dividing by $P-1$ and using $S_{m'}\le\omega(m')(P-1)$ gives $\varphi(m')P^{a-1}\le\alpha+|c|+\alpha\,\omega(m')\le L(1+\omega(m'))$, and $\omega(m')!\le\varphi(m')$ yields $\omega(m')\le\omega_0$ and $\varphi(m')P^{a-1}\le B$. For the bound on $P$, use instead Lemma \ref{lem:SleqPhi}: $c+\alpha S_{m'}\le|c|+\alpha\varphi(m')+\alpha$, so that
\[
\varphi(m')\bigl(P^{a-1}(P-1)-\alpha\bigr)\;\le\;|c|+\alpha P .
\]
If $P(P-1)\le\alpha$ then $P\le\alpha$. If $P(P-1)>\alpha$ then, as $P^{a-1}\ge P$ and $\varphi(m')\ge1$, we get $P^{2}-(\alpha+1)P-(\alpha+|c|)\le0$, hence $P\le\frac12\bigl(\alpha+1+\sqrt{(\alpha+1)^{2}+4L}\bigr)\le\alpha+1+\sqrt L$.

\emph{Counting.} For $a=1$ each $m'$ determines $P$, so there are at most $N_\varphi(B)\ll B$ sporadic solutions. For $a\ge2$ there are $\ll B$ choices of $m'$, $\ll L$ choices of $P$ and $\ll\log B$ choices of $a$, since $P^{a-1}\le B$. Hence $|F(c,\alpha)|\ll BL\log B\ll L^{2}(\log L)^{2}$, as $B\ll L\log L$.

\emph{Size.} Using $m'\le2\varphi(m')^{2}$ (Hardy--Wright \cite{HardyWright}), a solution with $a=1$ satisfies $n=m'P\le2B^{2}\cdot L(B+2)\ll LB^{3}\ll L^{4}(\log L)^{3}$, and a solution with $a\ge2$ satisfies $n=m'P^{a-1}\cdot P\le2\varphi(m')^{2}P^{a-1}\cdot P\le2B^{2}(\alpha+1+\sqrt L)\ll L^{3}(\log L)^{2}$.
\end{proof}

Thus $|F(c,\alpha)|$ and $\max F(c,\alpha)$ are bounded by polynomials in $|c|+\alpha$. The census of Theorem \ref{thm:census} ($|F|\le9$ for $L\le16$), extended by a direct scan of $n\le2\cdot10^{6}$ to $|c|\le40$, $\alpha\le12$ (where $\max|F|=19$, at $(36,12)$), suggests that the truth is far smaller; see Question \ref{q:growth}.

\begin{corollary}\label{cor:uniform}
\begin{enumerate}
\item[\textup{(i)}] If $\alpha\ge3$ is odd, there is no matched modulus, and for every $c$ the composite solution set of \eqref{eq:weighted} is finite, of cardinality $\ll(|c|+\alpha)^{2}\log^{2}(|c|+\alpha)$.
\item[\textup{(ii)}] For $\alpha=1$ and $c\le0$ one has $F(c,1)=\varnothing$: the equation $\varphi(n)=c+S_n$ has composite solutions only for $c=-1$ (the family $2p$) and for $c\ge1$.
\end{enumerate}
\end{corollary}

\begin{proof}
(i) A matched modulus $m$ satisfies $\varphi(m)=\alpha$, which is impossible for odd $\alpha\ge3$ since $\varphi(m)$ is even for $m\ge3$ and $\varphi(1)=\varphi(2)=1$; Theorem \ref{thm:uniform} then bounds the whole solution set. (ii) If $n$ is composite and not twice an odd prime, Theorem \ref{thm:RS} gives $\varphi(n)\ge1+S_n>c+S_n$ for $c\le0$; if $n=2p$ then $\varphi(n)=S_n-1$, so $c=-1$ and $n$ belongs to the matched family.
\end{proof}
}}

\subsection{Numerical census}

An exhaustive direct scan over composites $n\le2\cdot10^{5}$ and weights $|c|\le10$, $1\le\alpha\le6$ agrees with Theorems \ref{thm:families} and \ref{thm:sporadic} in every case, and shows in addition that all \emph{unmatched} pairs in the range have at most nine composite solutions up to $2\cdot10^{5}$; a small sample:
\begin{center}\small\setlength{\arraycolsep}{4pt}
$\begin{array}{c|l}
(c,\alpha) & \text{composite solutions } n\le2\cdot10^{5}\\\hline
(1,1) & 4,\ 12,\ 30\\
(0,2) & 4,\ 18\\
(2,2) & 8,\ 9,\ 24,\ 60,\ 70\\
(0,4) & 8,\ 36,\ 45,\ 50,\ 91,\ 105,\ 126\\
(6,6) & 27,\ 49,\ 72,\ 112,\ 180,\ 312,\ 380,\ 418,\ 462
\end{array}$
\end{center}
The parity constraints are visible throughout: every solution found for $(3,3)$ is even and every solution for $(0,3)$ is odd, in accordance with Theorem \ref{thm:parity}.

}
\section{\new{Simultaneous equations}}\label{sec:simultaneous}

\new{
Write $(E_\varphi)$, $(E_\tau)$, $(E_\sigma)$ for equation \eqref{eq:general} with $f=\varphi,\tau,\sigma$ respectively. Combining Theorems \ref{thm:RS}, \ref{thm:tau} and \ref{thm:sigma} yields complete answers for all conjunctions.

\begin{theorem}\label{thm:simultaneous}
\emph{(i)} The unique composite integer satisfying $(E_\varphi)$ and $(E_\tau)$ simultaneously is $n=4$.
\emph{(ii)} No integer $n\ge2$ satisfies $(E_\varphi)$, $(E_\tau)$ and $(E_\sigma)$ simultaneously; indeed already $(E_\sigma)$ alone has no composite solution, and no prime satisfies $(E_\tau)$.
\end{theorem}

\begin{proof}
(i) By Theorems \ref{thm:RS} and \ref{thm:tau} the composite solutions of $(E_\varphi)$ are $\{4,12,30\}$ and those of $(E_\tau)$ are the prime squares; the intersection is $\{4\}$, since neither $12$ nor $30$ is the square of a prime. One checks directly: $\varphi(4)=2=1+\varphi(2)$ and $\tau(4)=3=1+\tau(2)$.

(ii) For a prime $p$, $(E_\tau)$ reads $2=3$, false; combined with Theorem \ref{thm:sigma} for composites this settles the claim.
\end{proof}

Thus, imposing several of these equations at once collapses the solution set, first to $\{4\}$ and then to $\varnothing$.
}

\section{\new{The average size of the deficiency}}\label{sec:average}

\new{
Theorem \ref{thm:RS} says that the \emph{deficiency}
\[
\Delta(n)\;:=\;\varphi(n)-1-\sum_{p\mid n}\varphi(p)
\]
is nonnegative for composite $n\neq2p$ and vanishes only thrice. How large is it on average? The answer is: as large as $\varphi$ itself, the subtracted term being negligible at the scale $x^{2}$.

\begin{theorem}\label{thm:average}
Let $A_j=(-1)^{j-1}\dfrac{\mathrm{d}^{j-1}}{\mathrm{d}s^{j-1}}\Bigl(\dfrac{\zeta(s)}{s}\Bigr)\Big|_{s=2}$, so that $A_1=\dfrac{\pi^{2}}{12}$ and $A_2=\dfrac{\zeta(2)}{4}-\dfrac{\zeta'(2)}{2}$. For every fixed $M\ge1$,
\[
\sum_{n\le x}\Delta(n)\;=\;\frac{3}{\pi^{2}}\,x^{2}\;-\;\sum_{j=1}^{M}A_j\,\frac{x^{2}}{\log^{j}x}\;+\;O\!\left(\frac{x^{2}}{\log^{M+1}x}\right).
\]
In particular the mean value $\frac1x\sum_{n\le x}\Delta(n)\sim\frac{3}{\pi^{2}}x$.
\end{theorem}

\begin{proof}
Since $\varphi(p)=p-1$, we have $\sum_{p\mid n}\varphi(p)=\beta(n)-\omega(n)$, where $\beta(n)=\sum_{p\mid n}p$ is the sum of the distinct prime divisors of $n$. Hence
\[
\sum_{n\le x}\Delta(n)\;=\;\sum_{n\le x}\varphi(n)\;-\;\lfloor x\rfloor\;-\;\sum_{n\le x}\beta(n)\;+\;\sum_{n\le x}\omega(n).
\]
The first sum equals $\frac{3}{\pi^{2}}x^{2}+O(x\log x)$ \cite[Th.~330]{HardyWright}, and the last one is $O(x\log\log x)$ \cite[Th.~430]{HardyWright}. The leading term of $\sum_{n\le x}\beta(n)$, namely $\frac{\pi^{2}}{12}x^{2}/\log x$, is due to Alladi and Erd\H{o}s \cite{AlladiErdos77}, and the full expansion $\sum_{n\le x}\beta(n)=\sum_{j=1}^{M}A_jx^{2}/\log^{j}x+O(x^{2}/\log^{M+1}x)$ to De Koninck and Ivi\'c \cite{DeKoninckIvic84}. Combining these estimates proves the theorem.
\end{proof}

\begin{remark}
The leading term also follows from the exact identity $\sum_{n\le x}\sum_{p\mid n}\varphi(p)=\sum_{k\ge1}P(x/k)$, where $P(y)=\sum_{p\le y}(p-1)$, together with the prime number theorem. Numerically, $\bigl(\sum_{n\le x}\sum_{p\mid n}\varphi(p)\bigr)\log x/x^{2}$ equals $0.9153$, $0.8978$, $0.8858$ at $x=10^{5},10^{6},10^{7}$ (exact sieve computations), whereas the expansion $\sum_{j\le M}A_j/\log^{j-1}x$ truncated at $M=2$ gives $0.8989$, $0.8862$, $0.8771$ and at $M=3$ reduces the discrepancy at $x=10^{7}$ to $0.2\%$.
\end{remark}
}

\section{\C{Iterated totients on the prime divisors}}\label{sec:iterated}

\C{
In this section we replace the prime values $\varphi(p)$ in \eqref{eq:main} by their iterates and determine, for every $k\ge1$, all positive integers $n$ with
\begin{equation}\label{eq:iter}
\varphi(n)\;=\;1+\sum_{p\mid n}\varphi^{(k)}(p).
\end{equation}
The answer rests on Theorem \ref{thm:RS} and on the following monotonicity.

\begin{lemma}\label{lem:iter-mono}
For every $n\ge1$ the sequence $\bigl(\varphi^{(j)}(n)\bigr)_{j\ge1}$ is non-increasing. Moreover, for every prime $p\ge3$ and every $k\ge2$, $\varphi^{(k)}(p)\le\varphi(p-1)\le\frac{p-1}{2}$.
\end{lemma}

\begin{proof}
Since $\varphi(m)\le m$ for all $m\ge1$, we have $\varphi^{(j+1)}(n)=\varphi\bigl(\varphi^{(j)}(n)\bigr)\le\varphi^{(j)}(n)$. For $p\ge3$ prime, $\varphi^{(2)}(p)=\varphi(p-1)$ and $p-1$ is even, so $\varphi(p-1)\le(p-1)/2$; the first part gives $\varphi^{(k)}(p)\le\varphi^{(2)}(p)$ for $k\ge2$.
\end{proof}

Write $X:=\{1,2,4,12,30\}\cup\{p,\,2p:\ p\ \text{odd prime}\}$. For $n\notin X$, $n$ is composite and not twice an odd prime, and $n\notin\{4,12,30\}$; hence Theorem \ref{thm:RS} gives $\varphi(n)>1+\sum_{p\mid n}\varphi(p)\ge1+\sum_{p\mid n}\varphi^{(k)}(p)$ by Lemma \ref{lem:iter-mono}. Thus, for every $k\ge1$, every solution of \eqref{eq:iter} lies in $X$, and it remains to examine the elements of $X$.

\begin{theorem}\label{thm:iter}
The solutions of \eqref{eq:iter} are exactly:
\begin{enumerate}
\item[\emph{(i)}] $n\in\{1,4,12,30\}$ if $k=1$;
\item[\emph{(ii)}] $n\in\{1,3,4,10\}$ if $k=2$;
\item[\emph{(iii)}] $n\in\{1,3,4\}$ if $k\ge3$.
\end{enumerate}
\end{theorem}

\begin{proof}
The values involved are collected in the following table, where the right-hand side of \eqref{eq:iter} is denoted $R_k(n)$:
\begin{center}\small
$\begin{array}{c|c|c|c|c}
n & \varphi(n) & R_1(n) & R_2(n) & R_k(n),\ k\ge3\\\hline
1 & 1 & 1 & 1 & 1\\
2 & 1 & 2 & 2 & 2\\
4 & 2 & 2 & 2 & 2\\
12 & 4 & 4 & 3 & 3\\
30 & 8 & 8 & 5 & 4\\
3 & 2 & 3 & 2 & 2\\
6 & 2 & 4 & 3 & 3\\
10 & 4 & 6 & 4 & 3
\end{array}$
\end{center}
(For $n=1$ the sum is empty. For $k\ge3$, $\varphi^{(k)}(2)=\varphi^{(k)}(3)=1$ and $\varphi^{(k)}(5)=1$.) This settles $n\in\{1,2,4,12,30\}$ for all $k$, and $n\in\{6,10\}$ once these are shown to be the only candidates among the $2p$.

\emph{Case $k=1$.} For an odd prime $p$, $\varphi(p)=p-1<p=R_1(p)$, and $\varphi(2p)=p-1<p+1=R_1(2p)$. With the table, the solutions are $\{1,4,12,30\}$.

\emph{Odd primes, $k\ge2$.} Let $n=p$ be an odd prime. For $k=2$, \eqref{eq:iter} reads $p-1=1+\varphi(p-1)$, i.e.\ $\varphi(p-1)=p-2$, which forces $p-1$ to be prime; $p-1$ and $p$ being consecutive primes, $p=3$. For $k\ge3$, Lemma \ref{lem:iter-mono} gives $p-1=1+\varphi^{(k)}(p)\le1+\frac{p-1}{2}$, hence $p\le3$, and $p=3$ is a solution by the table.

\emph{Twice an odd prime, $k\ge2$.} Let $n=2p$ with $p$ an odd prime, so $\varphi(n)=p-1$ and $R_k(2p)=1+\varphi^{(k)}(2)+\varphi^{(k)}(p)=2+\varphi^{(k)}(p)$. For $k=2$ this gives $\varphi(p-1)=p-3$; since $\varphi(p-1)\le\frac{p-1}{2}$ we get $p\le5$. For $k\ge3$, Lemma \ref{lem:iter-mono} gives $p-1\le2+\frac{p-1}{2}$, again $p\le5$. Thus $n\in\{6,10\}$, and the table shows that $10$ is a solution for $k=2$ only, while $6$ never is.
\end{proof}

\begin{remark}
The solution sets shrink from $k=1$ to $k=2$ and stabilise \textcolor{blue}{for all $k\geq 3$}. The only solution lost between $k=2$ and $k=3$ is $n=10$, because $\varphi^{(3)}(5)=1<\varphi^{(2)}(5)=2$; the stabilisation reflects the fact that, for every prime $p\ge3$, $\varphi^{(2)}(p)=\varphi(p-1)$ is even, so the iterates beyond the second are at most half of it and the right-hand side of \eqref{eq:iter} can only reach $\varphi(n)$ for the smallest values of $n$. All three solution sets were confirmed by exhaustive computation up to $10^{5}$ for $k\le8$.
\end{remark}
}

\B{The weighted form of \eqref{eq:iter}, $\varphi(n)=c+\alpha\sum_{p\mid n}\varphi^{(k)}(p)$, behaves differently from the case $k=1$ treated in \S\ref{sec:weighted}: since $\varphi^{(k)}(p)$ is no longer linear in $p$, the matched families of Theorem \ref{thm:families} disappear, but infinite families of a new, conditional kind appear. The simplest instance is the following.}

\begin{lemma}\label{lem:safe}
\B{Let $p\ge7$ be a prime. Then $\varphi(p-1)=\frac{p-3}{2}$ if and only if $\frac{p-1}{2}$ is prime, i.e.\ $p$ is a safe prime.}
\end{lemma}

\begin{proof}
\B{Write $p-1=2^{e}u$ with $u$ odd, $e\ge1$. Then $\varphi(p-1)=2^{e-1}\varphi(u)$. If $e\ge2$, then $\varphi(p-1)\le2^{e-1}u=\frac{p-1}{2}\cdot\frac{\varphi(u)}{u}\le\frac{p-1}{2}$, with $\varphi(p-1)=\frac{p-3}{2}$ only if $2^{e-1}(u-\varphi(u))=1$, impossible for $e\ge2$. So $e=1$, $p-1=2u$, $\varphi(p-1)=\varphi(u)$, and $\varphi(u)=u-1$ holds if and only if $u=\frac{p-1}{2}$ is prime; since $p\ge7$, $u\ge3$ is odd.}
\end{proof}

\begin{proposition}\label{prop:safe}
\B{The composite solutions of $\varphi(n)=4\sum_{p\mid n}\varphi^{(2)}(p)$ are exactly
\[
\{8,\,24,\,60\}\;\cup\;\{3p,\ 4p:\ p\ \text{safe prime},\ p\ge7\}\;\cup\;\{6p:\ p=4q+1\ \text{prime},\ q\ \text{prime}\}.
\]
In particular, the solution set is infinite if and only if there are infinitely many primes $p$ with $\frac{p-1}{2}$ or $\frac{p-1}{4}$ prime; under Dickson's conjecture all three families are infinite.}
\end{proposition}

\begin{proof}
\B{Let $P$ be the largest prime factor of $n$ and write $n=mP^{a}$, $P\nmid m$. If $a\ge2$, the argument of Theorem \ref{thm:structure} (case $a\ge2$) applies verbatim, since $\varphi^{(2)}(p)\le\varphi(p)$, and the resulting finite search gives only $n=8$. Let $a=1$ and $m\ge2$. Writing $S^{(2)}_m=\sum_{q\mid m}\varphi(q-1)$, the equation reads
\[
\varphi(m)(P-1)\;=\;4S^{(2)}_m+4\varphi(P-1).
\]
Since $\varphi(P-1)\le\frac{P-1}{2}$, this gives $\varphi(m)\le2+\frac{4S^{(2)}_m}{P-1}$. If $\varphi(m)\ge4$ then $P-1\le2S^{(2)}_m\le2\sum_{q\mid m}(q-2)$, which bounds $P$ in terms of $m$ while $P$ exceeds every prime factor of $m$; the finitely many cases are $n=24=8\cdot3$ and $n=60=12\cdot5$, where equality $\varphi(m)=2+4S^{(2)}_m/(P-1)=4$ holds. Otherwise $\varphi(m)=2$, i.e.\ $m\in\{3,4,6\}$. For $m\in\{3,4\}$ one has $S^{(2)}_m=1$ and the equation becomes $\varphi(P-1)=\frac{P-3}{2}$, i.e.\ $P$ is a safe prime by Lemma \ref{lem:safe} (and $P\ge7$ since $P>3$ and $P=5$ fails). For $m=6$, $S^{(2)}_6=2$ and the equation becomes $\varphi(P-1)=\frac{P-5}{2}$; writing $P-1=2^{e}u$ as in Lemma \ref{lem:safe}, this forces $e=2$ and $\varphi(u)=u-1$, i.e.\ $P=4q+1$ with $q$ prime.}
\end{proof}

\B{The sets were confirmed by exhaustive computation up to $2\cdot10^{5}$: of the $1029$ composite solutions, $872$ are of the form $3p$ or $4p$ with $p$ safe, $154$ of the form $6p$ with $p=4q+1$, and $3$ are sporadic. Thus, for $k\ge2$, finiteness of the solution sets of the weighted iterated equation is not a theorem but is tied to the distribution of primes $p$ for which $\varphi(p-1)$ is prescribed, a question of Sophie Germain type.}

\section{\new{Numerical verification}}\label{sec:numerics}

\new{
All finite claims of this paper were verified by an exhaustive computation over $2\le n\le 2\cdot10^{5}$ using SymPy. The composite solution sets found were: $\{4,12,30\}$ for $(E_\varphi)$, in agreement with \cite{RanjanSingh25}; $\{p^{2}\le 2\cdot10^{5}\}$ and nothing else for $(E_\tau)$; the empty set for $(E_\sigma)$, for $J_2$, for the weighted equation with $(c,\alpha)=(1,2)$, and for the triple system; and $\{4\}$ for the pair $(E_\varphi)\wedge(E_\tau)$. The expansion of Theorem \ref{thm:average} was compared with exact sieve computations up to $x=10^{7}$, and the sporadic sets of Theorem \ref{thm:sporadic} were obtained by the bounded mechanism enumeration of Lemma \ref{lem:bounds}, independently cross-checked against the direct scan. \B{The solution sets of Theorem \ref{thm:iter} were checked for $n\le10^{5}$ and $1\le k\le8$.} \V{The census of Theorem \ref{thm:census} was cross-checked against a direct scan of all composites $n\le10^{6}$, and the values $|F(c,1)|$ for $1\le c\le2000$ quoted in Question \ref{q:values} were obtained by the bounded enumeration of Theorem \ref{thm:uniform}(i)--(ii).} The verification scripts are available from the authors.
}

\section{Concluding remarks and open problems}

\new{
The results assembled here suggest that the natural home of the Ranjan--Singh theorem is the family of prime--composite characterization problems, alongside Subbarao's $\{4,6,22\}$ \cite{Subbarao74} and, in the unsolved direction, Lehmer's problem \cite{Lehmer32}. We close with \V{four} problems.

\V{\begin{question}\label{q:growth}
Theorem \ref{thm:uniform} gives $|F(c,\alpha)|\ll L^{2}(\log L)^{2}$ with $L=|c|+\alpha$, whereas Theorem \ref{thm:census} gives $|F(c,\alpha)|\le9$ for $L\le16$, a direct scan of $n\le2\cdot10^{6}$ gives $|F(c,\alpha)|\le19$ for $|c|\le40$, $\alpha\le12$, and $|F(c,1)|\le9$ for all $1\le c\le2000$. Is $\sup_{c}|F(c,\alpha)|$ finite for each fixed $\alpha$? The maximizers observed, $(6,6)$, $(12,4)$, $(40,8)$, $(27,9)$, $(36,12)$, all satisfy $\alpha\mid c$, as one expects from \eqref{eq:mech}: when $\alpha$ divides $c$ and $\varphi(m')$, the divisibility obstruction in $P-1=(c+\alpha S_{m'})/(\varphi(m')-\alpha)$ disappears. Is $|F(c,\alpha)|$ governed by $\gcd(c,\alpha)$ and the divisor structure of $\alpha$?
\end{question}

\begin{question}\label{q:values}
For $\alpha=1$, equation \eqref{eq:weighted} reads $\Delta(n)=c-1$, with $\Delta$ the deficiency of \S\ref{sec:average}, so that $|F(c,1)|$ is the multiplicity of the value $c-1$ of $\Delta$ on composite integers outside the family $\{2p\}$, the analogue for $\Delta$ of the classical multiplicity problems for $\varphi$ \cite{Ford99}. For $1\le c\le2000$, $F(c,1)$ is empty for $341$ values of $c$, the first being $6,12,22,23,24,30,42,46,60,72$; the mean multiplicity is stable, close to $2$; and every multiplicity $0,1,\dots,9$ occurs, first at $c=6,2,4,1,3,27,15,403,1047,943$. Thus the analogue of Carmichael's conjecture fails for $\Delta$ ($F(c,1)$ is a singleton for $c=2,8,10$). Does every multiplicity occur, as Ford proved for $\varphi$? Does the set of $c$ with $F(c,1)=\varnothing$ have positive density? Its parity bias is striking: it contains about $30\%$ of the even values of $c$ against $5\%$ of the odd ones, an even $c$ requiring an odd $n$ by Theorem \ref{thm:parity}(iii).
\end{question}}

\begin{question}
\B{Theorem \ref{thm:iter} settles the iterated equation with $c=\alpha=1$, and Proposition \ref{prop:safe} shows that for $(c,\alpha)=(0,4)$ and $k=2$ the solution set is infinite if and only if there are infinitely many safe primes. Describe, for general $(c,\alpha)$ and $k\ge2$, the conditional families of the weighted iterated equation $\varphi(n)=c+\alpha\sum_{p\mid n}\varphi^{(k)}(p)$: they are of the form $\{mp\}$ with $\varphi(m)\le\alpha$ and $p$ subject to a condition $\varphi^{(k-1)}(p-1)=\ell p+r$, and one expects each of them to be infinite under Dickson's conjecture.}
\end{question}

\begin{question}
Theorem \ref{thm:average} describes the average of $\Delta(n)$. Determine the distribution of $\Delta(n)/\varphi(n)$, which is expected to tend to $1$ in mean.
\end{question}

Finally, apart from the squares of primes (OEIS A001248), none of the solution sets occurring here, in particular $\{4,12,30\}$, appears to be catalogued in the OEIS at the time of writing; submitting them would usefully anchor the problem in the literature.
}

\end{document}